\documentclass[12pt]{article}
\usepackage{amsmath,amssymb,amsthm}
\usepackage[a4paper, margin=2.5cm]{geometry}

\title{Polynomial contractions of $\mathbb C^d $ and degree growth}

\author{
Dmitrii Korshunov\thanks{%
IMJ-PRG, Sorbonne Université, Paris, France.
\texttt{korshunov@imj-prg.fr}}
}

\newtheorem{proposition}{Proposition}
\newtheorem{definition}{Definition}
\newtheorem{theorem}{Theorem}
\newtheorem{question}{Question}
\newtheorem{remark}{Remark}
\newcommand{\aut}{\mathrm{Aut\,}}

\begin{document}
\maketitle

\begin{abstract}
 We give a simple example of a polynomial contraction automorphism of $\mathbb C^d$, $ d\ge 3 $, with unbounded degree growth. Combined with Poincar\'e--Dulac theorem it provides an algebraic automorphism of $\mathbb C^d$, $ d\ge 3 $, which is holomorphically but not algebraically linearizable.
\end{abstract}

\section{Introduction}

This note answers a question of Verbitsky and Ornea in connection with their work on GAGA-type theorems for complex manifolds equipped with a holomorphic contraction \cite{OV}.

\begin{definition}
An automorphism $\gamma\in \aut\mathbb C^d$ is called a contraction if $ 0 $ is the unique fixed point and  $\lim_{n\to \infty}\gamma^n(x)=0$ for all $x\in \mathbb C^n$.
\end{definition}

\begin{definition}
Suppose a polynomial automorphism $ \gamma $ of $ \mathbb C^n $ is given by \\  $(x_1,\dots x_d)\mapsto (P_1(x_1,\dots x_d),\dots (P_d(x_1,\dots x_d))$. We define the degree of $ \gamma $ to be the largest among degrees of $ P_i $.
\end{definition}

\begin{question}
Is there a polynomial contraction of $ \mathbb C^d $, such that degree of $\gamma^n$ grows to infinity with $ n\to\infty $?
\end{question}

For each $ d\ge 3 $ we construct an example of such contraction, and  show that for $ d= 2 $ they do not exist. The property of an algebraic automorphism to be of bounded degree growth is an algebraic conjugation invariant. On the other hand, Poincar\'e-Dulac theorem applied to our examples guarantees that they are holomorphically conjugate to a linear map. Thus we obtain examples of automorphisms that are holomorphically conjugate to a linear map but not algebraically conjugate to any linear map.

\subsection*{Acknowledgement}

The author is grateful to Misha Verbitsky, Yulia Gorginyan, Egor Yasinsky, and Mikhail Zaidenberg for their help during the preparation of this note. The author was supported by the European Research Council (ERC) Grant No. 101045750 (HodgeGeoComb).
\section{Algebraic groups of automorphisms and bounded degree}

We will need the notion of strict algebraicity of a polynomial automorphism, which captures the property of having bounded degree growth.

\begin{definition}
Let $\gamma\in \aut \mathbb A^d$. If there exists an algebraic action of an algebraic group $\Phi: G\times\mathbb A^d\to\mathbb A^d$, such that $\gamma(\cdot)=\Phi(g,\cdot)$ for some $g\in G$, then $\gamma$ is called strictly algebraic.
\end{definition}

\begin{proposition}
A polynomial automorphism $\gamma\in \aut \mathbb A^d$ is strictly algebraic if and only if $ \sup_n \deg \gamma^n <+\infty $. 
\end{proposition}
\begin{proof}
Let $\aut_{\le e} $ be the (constructible) set of the space of all $d$-tuples of polynomials of degree $ \le e$, naturally identified with $ \mathbb A^{d\binom{d+e}{e}}$. If $\deg \gamma^n\le e$ for all $n\ge 0$, then take as $G$ the Zariski closure of the set $\{\gamma^n: n\ge 0\} $ inside $\aut_{\le e} $. It is left to note that a constructible group over $ \mathbb C $ is actually algebraic.

Another way to see it, is to take the ind-Zariski closure of $\{\gamma^n: n\ge 0\} $ in $\aut \mathbb A^d $ with its standard structure of an ind-group and see that when degree is bounded it lands to the finite dimensional ind-subgroup, that is an algebraic group. For a comprehensive treatment of groups of automorphisms as ind-groups see \cite{FK}.

To prove the converse, recall that an action of an algebraic group $G$ on the coordinate ring of an affine variety is locally finite \cite[Proposition 1.9]{borel}. That is, each polynomial $f \in \mathbb{C}[x_1,\dots x_d]$ belongs to a finite dimensional $G$-invariant subspace $ G\cdot f $ (indeed, its dimension is bounded from above by $\dim G$). Now choose the coordinate polynomials $x_1,\dots x_d$, and take the smallest subspace of the filtration in which all $G\cdot x_1,\dots G\cdot x_d$ lie.

\end{proof}

\begin{remark}
If $ \gamma $ is strictly algebraic then $F\circ \gamma\circ F^{-1}$ is also strictly algebraic for any polynomial automorphism $F\in \aut \mathbb A^d  $.
\end{remark}

\section{Polynomial contractions}

\begin{theorem}
Any contraction automorphism $\gamma$ of $\mathbb C^2$ has bounded degree growth.
\end{theorem}
\begin{proof}
By results of Milnor and Friedland \cite{FM}, any $\gamma\in\aut \mathbb C^2$ is conjugate to either
\begin{enumerate}
\item Affine, with constant degree-growth.
\item Elementary, that is of the form $e(x,y)=(ax+P(y),by+c)$, with constant degree growth.
\item Loxodromic, i.e. a finite composition of Hénon automorphisms $h_i(x,y)=(ay+P_i(x),x)$, $ \deg P_i\ge 2 $, with exponential degree growth .
\end{enumerate}

On the other hand, by a theorem of Smillie \cite{smillie} and Theorem 4.1 in \cite{FM} the topological entropy of a loxodromic automorphism equals the logarithm of its degree. But a contraction does not have periodic points except the origin, hence its topological entropy must be zero. Hence a loxodromic automorphism cannot be a contraction.
\end{proof}

\begin{remark}
Note that Cantat and Dujardin \cite{cantat} prove that any two holomorphically conjugare loxodromic automorphisms of $\mathbb C^2$ are algebraically conjugate.
\end{remark}

Contrary to the two-dimensional case, in dimension $ d>2 $ there are polynomial contractions with unbounded growth of degree. This example is an affine-triangular automorphism in the sence of \cite{blanc}.

\begin{theorem}
An automorphism of $\mathbb A^3$ given by
$$\gamma: (x,y,z)\mapsto  \left(\lambda_1(y+xz),\lambda_2 x,\lambda_3 z\right), 0<\lambda_i<\frac{1}{2} $$
is a contraction and has linear degree growth.
\end{theorem}

\begin{proof}
Let $(x_0,y_0,z_0)\in \mathbb A^3$ be an arbitrary point.
Denote $x_n:=\gamma^n(x_0)$. Expanding the recursive relations one obtains

\begin{enumerate}
\item $  z_n=\lambda_3^n z_0$
\item $ y_n=\lambda_2 x_{n-1} $
\item $ x_n=\lambda_1(y_{n-1}+x_{n-1}z_{n-1})= \lambda_1(\lambda_2(x_{n-2})+\lambda_3^{n-1}z_0 x_{n-1})$
\end{enumerate}

For $\lambda_3^{n-1}z_0<\lambda_2$ one has $  |x_n|< \lambda_1 \lambda_2(|x_{n-2}|+|x_{n-1}|)$. Now note that $ \lambda_2( |x_{n-2}|+|x_{n-1}|)<\max \{|x_{n-2}|,|x_{n-1}|\} $, hence $ |x_n|<\lambda_1\max \{|x_{n-1}|,|x_{n-2}|\}$. This sequence converges to zero as $ \lambda_1^{n-2}\max\{x_1,x_0\}$. Thus $\lim_{n\to +\infty}|x_n|=\lim_{n\to +\infty}|y_n|=\lim_{n\to +\infty}|z_n|=0$, so we proved that $ \gamma $ is a contraction.

To compute the growth rate of degree, denote $ \gamma^n:=(F_n,G_n,H_n)$. Now
$ \deg \gamma^{n+1}=\max \{\deg F_{n+1},\deg G_{n+1},\deg H_{n+1}\}$, where $\deg H_{n+1}=1$, $ \deg F_{n+1} =\max\{\deg F_{n}+1, \deg F_{n-1}\}$, $ \deg G_{n+1}=\deg F_n$. Hence $  \deg \gamma^{n+1}=\deg \gamma^n+1$, and finally $ \deg \gamma^n = n+1 $.

\end{proof}

\begin{remark}
The theorem still holds and its proof goes through for any norm on $ \mathbb A^3 $ over a normed field.
\end{remark}

\begin{remark}
This example can be used to produce a contraction with unbounded degree growth in any higher dimension. Indeed, define it as $ \gamma $ on the first three coordinates, and by scaling by $\lambda_i$, $ i>3 $, on the rest, with some $0 <\lambda_i< 1$.
\end{remark}

\begin{remark}
The Jacobian matrix of this map is
$$D_0\gamma =\begin{pmatrix}
\lambda_1 z & \lambda_1 & \lambda_1 x\\
\lambda_2 & 0 & 0\\
0 & 0 & \lambda_3
\end{pmatrix} $$

Its determinant is constant and equals $ -\lambda_1 \lambda_2 \lambda_3 $. The eigenvalues of the differential in $0$ are $\pm \sqrt{\lambda_1\lambda_2}$ and $\lambda_3$.
\end{remark}

\section{Holomorphically linearizable automorphisms of $\mathbb C^n$}

In this section we will recall the Poincar\'e--Dulac theorem \cite[Theorem 27.38]{OVM}. 

\begin{definition}
\label{resonance}
A linear map with eigenvalues $\alpha_1,\dots \alpha_d $ is called resonant if for some $ i $ one has $ \alpha_i=\alpha_1^{m_1}\dots \alpha_d^{m_d} $, where all $ m_j\ge 0$ and $m_1+\dots +m_d \ge  2$. 
\end{definition}

\begin{theorem}
Let $\gamma$ be an invertible holomorphic contraction of $\mathbb C^n$ centered at $0$, with $n \ge 3$. Assume that the differential
$D_0\gamma \in GL(T_0\mathbb{C}^n)$
is not resonant. Then there exists a holomorphic automorphism
$U : \mathbb{C}^n \to \mathbb{C}^n$ 
such that
\[
U \circ \gamma \circ U^{-1}
\]
is linear.
\end{theorem}

Recall that the eigenvalues of the differential at zero of the polynomial contractions constructed in the previous section are $\sqrt{\lambda_1\lambda_2}, -\sqrt{\lambda_1\lambda_2}, \lambda_3,\dots \lambda_d $. If we chose all $ \lambda_i $ algebraically independent, then eigenvalues cannot satisfy the relations of Definition \ref{resonance}. Thus we proved

\begin{theorem}
Let $ \gamma_d $ be an automorphism of $\mathbb C^d$, $d\ge 3$

$$\gamma: (x_1,x_2,x_3,\dots x_d)\mapsto  \left(\lambda_1(x_2+x_1x_3),\lambda_2 x_1,\lambda_3 x_3,\dots \lambda_d x_d \right), 0<\lambda_i<\frac{1}{2} $$
where $ \lambda_i $ are algebraically independent.  Then
\begin{enumerate}

\item There exists a holomorphic diffeomorphism $ U $ such that $ U\circ\gamma_d\circ U^{-1} $ is linear.
\item There are no polynomial automorphisms with this property.
\end{enumerate}
\end{theorem}

\bibliographystyle{plain}
\bibliography{biblio}

\end{document}